\documentclass{article}
\usepackage{latexsym}
  \usepackage[all]{xy}
  \usepackage{amsfonts}
  \usepackage{amsthm}
  \usepackage{amsmath}
  \usepackage{amssymb}
  \usepackage{pifont}
  \usepackage{a4wide}
  \usepackage[T1]{fontenc}
\usepackage{graphicx}

\newtheorem{theorem}{Theorem}[section]

\newtheorem{proposition}[theorem]{Proposition}

\newtheorem{lemma}[theorem]{Lemma}
\begin{document}
{\center{{\bf{Recursive constructions of $k$-normal polynomials over
finite fields}}}\vspace{.5cm}
{\center{Mahmood Alizadeh}\\
Department of Mathematics,
Ahvaz Branch, Islamic Azad University, Ahvaz, Iran\\
 alizadeh@iauahvaz.ac.ir\\
\vspace{.5cm}
Saeid Mehrabi\\
Department of Mathematics, Farhangian University, Tehran, Iran\\
{\hspace{5.5cm}{saeid\_mehrabi@yahoo.com}}}}\\

\begin{abstract}
The paper is devoted to produce infinite sequences of $k$-normal
polynomials $F_{u}(x)\in \mathbb{F}_{q}[x]$ of degrees $np^{u} ~
(u\geq 0)$, for a suitably chosen initial $k$-normal polynomial
$F_{0}(x)\in \mathbb{F}_{q}[x]$ of degree $n$ over $\mathbb{F}_{q}$
 by iteratively applying the
transformation $x\rightarrow \frac{x^p-x}{x^p-x+\delta}$, where
$\delta\in \mathbb{F}_{q}$ and $0\leq k\leq n-1$.
\end{abstract}



\section{Introduction}\label{sec11}

Let $\mathbb{F}_{q}$ be the Galois field of order $q=p^{m}$, where
$p$ is a prime and $m$ is a natural number, and
$\mathbb{F}^{\ast}_{q}$ be its multiplicative group. A $normal~
basis$ of $\mathbb{F}_{q^n}$ over $\mathbb{F}_{q}$ is a basis of the
form $ N={\{\alpha, \alpha^{q}, ... , \alpha^{q^{n-1}}\}}$,  i.e. a
basis that consists of the algebraic conjugates
 of a
fixed element $\alpha\in{\mathbb{F}^{\ast}_{q^n}}$. Such an element
$\alpha\in\mathbb{F}_{q^n}$ is said to generate a normal
 basis for $\mathbb{F}_{q^n}$ over $\mathbb{F}_{q}$, and for convenience, is called a $normal~element$.

A monic irreducible polynomial $F(x)\in\mathbb{F}_{q}[x]$ is called
{\it{normal polynomial}} or {\it{$N$-polynomial}} if its roots are
linearly independent over $\mathbb{F}_{q}$. Since the elements in a
normal basis are exactly the roots of some N-polynomials, there is a
canonical one-to-one correspondence between N-polynomials and normal
elements. Normal bases have many applications, including  coding
theory, cryptography and computer algebra systems. For further
details, see \cite{8}.

 Recently, the $k$-normal elements over finite fields are defined and characterized by
Huczynska et al \cite{7}. For each $0\leq k\leq n-1$, the element
$\alpha \in \mathbb{F}_{q^n}$ is called a $k$-normal element if
$deg(gcd(x^n-1, \sum_{i=0}^{n-1}{\alpha^{q^i}}x^{n-1-i}))=k$.

 By analogy with the case of
normal polynomials, a monic irreducible polynomial $P(x)\in
\mathbb{F}_{q}[x]$ of degree $n$ is called a $k$-normal polynomial
(or $N_k$-polynomial) over $\mathbb{F}_{q}$ if its roots are
$k$-normal elements of $\mathbb{F}_{q^n}$ over $\mathbb{F}_{q}$.
Here, $P(x)$ has $n$ distinct conjugate roots, of which $(n-k)$ are
linearly independent. Recall that an element $\alpha \in
\mathbb{F}_{q^n}$ is called a proper element of $\mathbb{F}_{q^n}$
over $\mathbb{F}_{q}$ if $\alpha\notin \mathbb{F}_{q^v}$ for any
proper divisor $v$ of $n$. So, the element $\alpha \in
\mathbb{F}_{q^n}$ is a proper $k$-normal element of
$\mathbb{F}_{q^n}$ over $\mathbb{F}_{q}$ if $\alpha$ is a $k$-normal
and proper element of $\mathbb{F}_{q^n}$ over $\mathbb{F}_{q}$.

Using the above mention, a normal polynomial (or element) is a
0-normal polynomial (or element). Since the proper $k$-normal
elements of $\mathbb{F}_{q^n}$ over $\mathbb{F}_{q}$ are the roots
of a $k$-normal polynomial of degree $n$ over $\mathbb{F}_{q}$,
hence the $k$-normal polynomials of degree $n$ over $\mathbb{F}_{q}$
is just another way of describing the proper $k$-normal elements of
$\mathbb{F}_{q^n}$ over $\mathbb{F}_{q}$. Some results regarding
constructions of spacial sequences of $k$-normal polynomials over
$\mathbb{F}_{q}$, in the cases $k=0$ and $1$ can be found in
\cite{1x,3,4,9,10} and \cite{5}, respectively.

 In this paper, a
recursive method for constructing the $k$-normal polynomials of
higher degree from a given $k$-normal polynomial of degree $n$ over
$\mathbb{F}_{q}$ for each $0\leq k\leq n-1$ is given.

\section{Preliminary notes}\label{sec11z}
We use the definitions, notations and results given by Huczynska
\cite{7}, Gao \cite{6} and Kyuregyan \cite{9,10}, where similar
problems are considered.
 We need the following results for our further study.

The trace of $\alpha$ in $\mathbb{F}_{q^n}$ over $\mathbb{F}_{q}$,
is given by
$Tr_{\mathbb{F}_{q^n}|{\mathbb{F}_{q}}}(\alpha)=\sum_{i=0}^{n-1}{\alpha^{q^i}}$.
For convenience, $Tr_{\mathbb{F}_{q^n}|{\mathbb{F}_{q}}}$ is denoted
by $Tr_{q^n|{q}}$.

\begin{proposition}\cite{5}\label{thm321k}
Let $n=n_{1}p^e$, for some $e\geq1$ and $a,b\in
\mathbb{F}^{\ast}_{q}$. Then the element $\alpha$ is a proper
$k$-normal element of $\mathbb{F}_{q^n}$ over $\mathbb{F}_{q}$ if
and only if $a+b\alpha$ is a proper $k$-normal element of
$\mathbb{F}_{q^n}$ over $\mathbb{F}_{q}$.
\end{proposition}

 Let $p$ denote the characteristic of $\mathbb{F}_{q}$ and let $n=n_{1}p^{e
}=n_{1}t$,  with $gcd(p,n_{1})=1$ and suppose that
 $x^{n}-1$ has the following factorization in $\mathbb{F}_{q}[x]:$
\begin{equation}
x^{n}-1=(\varphi_{1}(x)\varphi_{2}(x)\cdots\varphi_{r}(x))^{t},\label{eq310}
\end{equation}
where $\varphi_{i}(x)\in \mathbb{F}_{q}[x]$ are the distinct
irreducible factors of $x^{n}-1$. For each $s$, $0\leq s<n$, let
there is a $u_s>0$ such that $R_{s,1}(x)$, $R_{s,2}(x)$, $\cdots$,
$R_{s,{u_s}}(x)$ are all of the $s$ degree polynomials dividing
$x^{n}-1$. So, from (\ref{eq310}) we can write
$R_{s,i}(x)=\prod_{j=1}^{r}{{\varphi_{j}}^{t_{ij}}(x)}$, for each
$1\leq i\leq u_{s}$, $0\leq t_{ij}\leq t$.
  Let
\begin{equation}
\phi_{s,i}(x)=\frac{x^{n}-1}{R_{s,i}(x)}, \label{eq31d1}
\end{equation}
for $1\leq i \leq u_{s} $. Then, there is a useful characterization
 of the $k$-normal polynomials of degree $n$ over $\mathbb{F}_{q}$ as follows.

\begin{proposition} \cite{5}\label{pro3122}
Let $F(x)$ be an irreducible polynomial of degree $n$ over
$\mathbb{F}_{q}$ and $\alpha$ be a root of it. Let $x^{n}-1$ factor
as (\ref{eq310}) and let $\phi_{s,i}(x)$ be as in (\ref{eq31d1}).
Then $F(x)$ is a $N_{k}$-polynomial over $\mathbb{F}_{q}$ if and
only if, there is $j$, $1\leq j\leq u_k$, such that

$$L_{\phi_{k,{j}}}(\alpha)=0, $$
and also
$$L_{\phi_{s,{i}}}(\alpha)\neq 0,$$
for each $s$, $k<s<n$, and $1\leq i \leq u_{s} $, where $u_{s}$ is
the number of all $s$ degree polynomials dividing $x^{n}-1$ and
 $L_{\phi_{s,{i}}}(x)$ is the linearized polynomial defined by
$$L_{\phi_{s,{i}}}(x)=\sum_{v=0}^{n-s}t_{iv}x^{q^{v}} ~ if ~
\phi_{s,{i}}(x)=\sum_{v=0}^{n-s}t_{iv}x^{v}.$$
\end{proposition}

 The following propositions
are useful for constructing $N_k$-polynomials over $\mathbb{F}_{q}$.

\begin{proposition}\cite{2}\label{prothe1}
Let $x^{p}-\delta_{2}x+\delta_{0}$ and
$x^{p}-\delta_{2}x+\delta_{1}$ be relatively prime polynomials in
$\mathbb{F}_{q}[x]$ and $P(x)=\sum_{i=0}^n{{c_{i}x^{i}}} $ be an
irreducible polynomial of degree $n\geq2$ over $\mathbb{F}_{q }$,
and let $ \delta_{0}, \delta_{1}\in\mathbb{F}_{q }$, $
\delta_{2}\in{\mathbb{F}^{*}_{q}}$, $(\delta_{0},\delta_{1})\neq
(0,0)$. Then
$$ F(x)={(x^{p}-\delta_{2}x+\delta_{1})^{n}}P\left(\frac{x^{p}-\delta_{2}x+\delta_{0}}{x^{p}-\delta_{2}x+ \delta_{1}}\right)$$
is an irreducible polynomial of degree $np$ over $\mathbb{F}_{q}$ if
and only if ${\delta_{2}}^{\frac{q-1}{p-1}}=1$ and
$$ Tr_{q|{p}}\left(\frac{1}{A^p}\left((\delta_{1}-\delta_{0})\frac{P'(1)}{P(1)}-n\delta_{1}\right)\right)\neq0,$$
where $A^{p-1}=\delta_{2}$, for some $A\in \mathbb{F}^{*}_{q }.$
\end{proposition}

\begin{proposition}\label{the4b}\cite{1}
Let $x^{p}-x+\delta_{0}$ and $x^{p}-x+\delta_{1}$ be relatively
prime polynomials in $\mathbb{F}_{q}[x]$ and let $ P(x)$  be an
irreducible polynomial of degree $ n\geq2 $ over $ \mathbb{F}_{q} $,
and $0\neq\delta_{1},\delta_{0}\in \mathbb{F}_{p} $, such that $
\delta_{0}\neq \delta_{1} $. Define
$$F_{0}(x)=P(x)\hspace{7.6cm}$$
$$ F_{k}(x)={(x^{p}-x+\delta_{1})^{t_{k-1}}}F_{k-1}\left(\frac{x^{p}-x+\delta_{0}}{x^{p}-x+\delta_{1}}\right), \hspace{1.4cm} k\geq 1$$
where $t_{k}=np^{k}$ denotes the degree of $ F_{k}(x) $. Suppose
that
$$  Tr_{q|{p}}\left(\frac{(\delta_{1}-\delta_{0})F_{0}'(1)-n\delta_{1}F_{0}(1)}{F_{0}(1)}\right)\cdot Tr_{q|{p}}\left(
\frac{(\delta_{1}-\delta_{0})F_{0}'(\frac{\delta_{0}}{\delta_{1}})+n\delta_{1}F_{0}(\frac{\delta_{0}}{\delta_{1}})}{F_{0}(\frac{\delta_{0}}{\delta_{1}})}\right)\neq0.$$
Then ${(F_{k}(x))}_{k\geq0}$ is a sequence of irreducible
polynomials over $\mathbb{F}_{q}$ of degree $t_{k}=np^{k}$, for
every $k\geq0$.
\end{proposition}

\begin{lemma}\label{lem21}
Let $\gamma$ be a proper element of $\mathbb{F}_{q^n}$ and
$\theta\in \mathbb{F}^*_{p}$, where $q=p^m$, ($m\in \mathbb{N}$).
Then we have
\begin{equation}
\sum_{j=0}^{p-1}{{\frac{1}{\gamma+j\theta}}}=-\frac{1}{\gamma^{p}-\gamma}.
 \label{eq3350}
\end{equation}

\end{lemma}

\begin{proof}
By observing that
\begin{align}
\sum_{j=0}^{p-1}{{\frac{1}{\gamma+j\theta}}}&=\frac{1}{\gamma^{p}-\gamma}
\left(\sum_{j=0}^{p-1}{\frac{\gamma^{p}-\gamma}{\gamma+j\theta}}\right),
 \label{eq3346}
\end{align}
it is enough to show that
\begin{align}
\sum_{j=0}^{p-1}{\frac{\gamma^{p}-\gamma}{\gamma+j\theta}}&=-1.
\end{align}
We note that
\begin{align}
\sum_{j=0}^{p-1}{\frac{\gamma^{p}-\gamma}{\gamma+j\theta}}&=\sum_{j=0}^{p-1}\left({{(\gamma+j\theta)}^{p-1}-1}\right)\nonumber\\
&=\sum_{j=0}^{p-1}{{(\gamma+j\theta)}^{p-1}}\nonumber\\
&=\sum_{j=1}^{p-1}{{\theta^j}}\binom{p-1}{j}\left(\sum_{i=1}^{p-1}{i}^j\right),\label{eq3348dss}
\end{align}
where
$$\binom{p-1}{j}=\frac{(p-1)!}{(p-1-j)!j!},~~~ j\in  \mathbb{F}^{*}_{p}.$$
On the other side, we know that
\begin{equation}
\sum_{i=1}^{p-1}{i^j}=
\begin{cases}
  ~0  ~\pmod{p},~~\mbox{if }~\mbox{p-1}\nmid j\\-1\pmod{p},~\mbox{if }\mbox{
  p-1}\mid j
\end{cases}
\label{eq3348ds}
\end{equation}
and also $\theta^{p-1}=1$. Thus by \eqref{eq3348dss} and
\eqref{eq3348ds}, the proof is completed.

\end{proof}

\section{Recursive construction $N_{k}$-polynomials}\label{sec3}
In this section we establish theorems which will show how
propositions \ref{pro3122}, \ref{prothe1} and \ref{the4b} can be
applied to produce infinite sequences of $N_{k}$-polynomials over
$\mathbb{F}_{q}$. Recall that, the polynomial $P^{*}(x)=x^n
P\left(\frac{1}{x}\right)$ is called the reciprocal polynomial of
$P(x)$, where $n$ is the degree of $P(x)$. In the case $k=0$, some
similar results of the following theorems has been obtained in
\cite{1x}, \cite{3} and (\cite{4}, Theorems 3.3.1 and 3.4.1). We use
of an analogous technique to that used in the above results, where
similar problems are considered.

\begin{theorem}\label{thm32uu3}
Let $P(x)=\sum_{i=0}^{n}c_{i}x^{i}$ be an $N_{k}$-polynomial of
degree $n$ over $\mathbb{F}_{q}$, for each $n=rp^e$, where $e\in
\mathbb{N}$ and $r$ equals 1 or is a prime different from $p$ and
$q$ a primitive element modulo $r$.
 Suppose that $\delta\in\mathbb{F}^{*}_{q}$ and
\begin{equation}
 F(x)={(x^{p}-x+\delta)}^{n}P^{*}\left(\frac{x^{p}-x}{x^{p}-x+ \delta}\right).
 \label{eq330qsas}
\end{equation}
Then $F^{*}(x)$ is an $N_{k}$-polynomial of degree $np$ over
$\mathbb{F}_{q}$ if $k< p^e$ and
$$Tr_{q|{p}}\left(\delta\frac{{P^{*}}^{'}(1)}{P^{*}(1)}\right)\neq0.$$
\end{theorem}
\begin{proof}
 Since $P^{*}(x)$ is an irreducible polynomial over
$\mathbb{F}_{q}$, so Proposition \ref{prothe1} and theorem's
hypothesis imply that
$F(x)$ is irreducible over $\mathbb{F}_{q}$.\\
Let $\alpha\in\mathbb{F}_{q^n}$ be a root of $P(x)$. Since $P(x)$ is
an $N_{k}$-polynomial of degree $n$ over $\mathbb{F}_{q}$ by
theorem's hypothesis, then $\alpha\in\mathbb{F}_{q^n}$ is a proper
$k$-normal element over $\mathbb{F}_{q}$.

 Since $q$ is a primitive
modulo $r$, so in the case $r>1$ the polynomial $x^{r-1}+\cdots+x+1$
is irreducible over $\mathbb{F}_{q}$. Thus $x^n-1$ has the following
factorization in $\mathbb{F}_{q}[x]$:

\begin{equation}
x^n-1=(\varphi_{1}(x)\cdot \varphi_{2}(x))^{t}, \label{eq3aww33}
\end{equation}
 where $\varphi_{1}(x)=x-1$, $\varphi_{2}(x)=x^{r-1}+ \cdots +x+1$
 and
 $t=p^e$.

Letting that for each $0\leq s<n$ and $1\leq i \leq u_{s} $,
$R_{s,i}(x)$ is the $s$ degree polynomial dividing $x^{n}-1$, where
$u_{s}$ is the number of all $s$ degree polynomials dividing
$x^{n}-1$. So, from (\ref{eq3aww33}), we can write
$R_{s,i}(x)={(x-1)}^{s_{1,i}}\cdot
{(x^{r-1}+\cdots+x+1)}^{s_{2,i}}$, where
$s=s_{1,i}+s_{2,i}\cdot{(r-1)}$ for each $0\leq s_{1,i}, s_{2,i}\leq
t$, except when $s_{1,i}= s_{2,i}= t$. So, we have

\begin{equation}\label{eq31w1}
\phi_{s,i}(x)=\frac{x^{n}-1}{R_{s,i}(x)}=\frac{x^{n}-1}{{(x-1)}^{s_{1,i}}\cdot
{(x^{r-1}+\cdots+x+1)}^{s_{2,i}}}=\sum_{v=0}^{n-s}{t_{s,i,v}{x^{v}}}.
\end{equation}

Since  $P(x)$ is an $N_{k}$-polynomial of degree $n$ over
$\mathbb{F}_{q}$, so by Proposition \ref{pro3122}, there is a $j$,
$1\leq j\leq u_{k}$, such that

\begin{equation}\label{eq31w1mp}
L_{\phi_{k,{j}}}(\alpha)=0,
\end{equation}
 and also
 \begin{equation}\label{eq31w1np}
L_{\phi_{s,{i}}}(\alpha)\neq 0,
\end{equation}
for each $k<s<n$ and $1\leq i \leq u_{s} $. Further, we proceed by
proving that $F^{*}(x)$ is a $k$-normal polynomial. Let $\alpha_{1}$
be a root of $F(x)$. Then $\beta_{1}=\frac{1}{\alpha_{1}}$ is a root
of its reciprocal polynomial $F^{*}(x)$. Note that by
(\ref{eq3aww33}), the polynomial $x^{np}-1$ has the following
factorization in $\mathbb{F}_{q}[x]$:

 \begin{equation}\label{eq31w1npo}
x^{np}-1={\left(\varphi_{1}(x)\cdot \varphi_{2}(x)\right)}^{pt},
\end{equation}
 where $\varphi_{1}(x)=x-1$, $\varphi_{2}(x)=x^{r-1}+ \cdots +x+1$
 and
 $t=p^e$.

Letting that for each $0\leq s'<np$ and $1\leq i' \leq u'_{s'} $,
$R'_{s',i'}(x)$ is the $s'$ degree polynomial dividing $x^{np}-1$,
where $u'_{s'}$ is the number of all $s'$ degree polynomials
dividing $x^{np}-1$. So, from (\ref{eq31w1npo}) we can write
$R'_{s',i'}(x)={(x-1)}^{s'_{1,i'}}\cdot
{(x^{r-1}+\cdots+x+1)}^{s'_{2,i'}}$, where
$s'=s'_{1,i'}+s'_{2,i'}\cdot{(r-1)}$ for each $0\leq s'_{1,i'},
s'_{2,i'}\leq pt$, except when $s'_{1,i'}= s'_{2,i'}= pt$. Therefore
by considering

\begin{equation}
H'_{s',i'}(x)=\frac{x^{np}-1}{R'_{s',i'}(x)}, \label{eq31w11ds}
\end{equation}
and Proposition \ref{pro3122}, $F^{*}(x)$ is an $N_{k}$-polynomial
of degree $np$ over $\mathbb{F}_{q}$ if and only if there is a $j'$,
$1\leq j'\leq u'_{k}$, such that

$$L_{H'_{k,{j'}}}(\beta_{1})=0, $$
 and also
$$L_{H'_{s',{i'}}}(\beta_{1})\neq 0,$$
for each $k<s'<np$ and $1\leq i' \leq u'_{s'}$ . Consider

 \begin{align}
H_{s,i}(x)&=\frac{x^{np}-1}{R_{s,i}(x)}\nonumber\\
&=\frac{x^{n}-1}{R_{s,i}(x)}\left(\sum_{j=0}^{p-1}x^{jn}\right),
\label{555ss}
\end{align}
for each $0\leq s<n$ and $1\leq i \leq u_{s} $. By (\ref{eq31w1}) we
obtain

$$H_{s,i}(x)=\phi_{s,i}(x)\left(\sum_{j=0}^{p-1}x^{jn}\right)=\sum_{v=0}^{n-s}{t_{s,i,v}}\left(\sum_{j=0}^{p-1}x^{jn+v}\right).$$
It follows that
$$L_{H_{s,i}}(\beta_{1})=\sum_{v=0}^{n-s}t_{s,i,v}\left(\sum_{j=0}^{p-1}\left(\beta_{1}\right)^{p^{jmn}}\right)^{p^{mv}},$$
or
\begin{equation}
L_{H_{s,i}}(\beta_{1})=L_{H_{s,i}}\left(\frac{1}{\alpha_{1}}\right)=
\sum_{v=0}^{n-s}t_{s,i,v}\left(\sum_{j=0}^{p-1}\left(\frac{1}{\alpha_{1}}\right)^{p^{jmn}}\right)^{p^{mv}}.
\label{eq339sa}
\end{equation}
From (\ref{eq330qsas}), if $\alpha_{1}$ is a zero of $F(x)$, then
$\frac{{\alpha_{1}}^{p}-{\alpha_{1}}+\delta}{{\alpha_{1}}^{p}-{\alpha_{1}}}$
is a zero of $P(x)$, and therefore it may assume that
$$\alpha=\frac{{\alpha_{1}}^{p}-{\alpha_{1}}+\delta}{{\alpha_{1}}^{p}-{\alpha_{1}}},$$
or
\begin{equation}
\frac{\alpha-1}{\delta}={({\alpha_{1}}^{p}-{\alpha_{1}})}^{-1}.
 \label{eq3340sa}
\end{equation}
Now, by (\ref{eq3340sa}) and observing that $P(x)$ is an irreducible
polynomial of degree $n$ over $\mathbb{F}_{q}$, we obtain
\begin{equation}
\frac{\alpha-1}{\delta}={(\frac{\alpha-1}{\delta})}^{p^{mn}}={({\alpha_{1}}^{p^{mn+1}}-{\alpha_{1}}^{p^{mn}})}^{-1}.
 \label{eq3341sa}
\end{equation}
It follows from (\ref{eq3340sa}) and (\ref{eq3341sa}) that
\begin{equation}
{({\alpha_{1}}^{p^{mn+1}}-{\alpha_{1}}^{p^{mn}})}^{-1}={({\alpha_{1}}^{p}-{\alpha_{1}})}^{-1}.
\label{eq3342}
\end{equation}
Also observing that $F(x)$ is an irreducible polynomial of degree
$np$ over $\mathbb{F}_{q}$, we have
 ${({\alpha_{1}}^{p}-{\alpha_{1}})}\neq0$ and
 ${({\alpha_{1}}^{p^{mn+1}}-{\alpha_{1}}^{p^{mn}})}\neq0$. Hence by
 (\ref{eq3342})
\begin{equation}
({{\alpha_{1}}^{p^{mn}}-{\alpha_{1}})}^{p}=({{\alpha_{1}}^{p^{mn}}-{\alpha_{1}})}.
\label{eq3343sa}
\end{equation}
It follows from (\ref{eq3343sa}) that
${\alpha_{1}}^{p^{mn}}-{\alpha_{1}}=\theta \in \mathbb{F}^*_{p}$.
Hence ${\alpha_{1}}^{p^{mn}}={\alpha_{1}}+\theta$ and

$${\alpha_{1}}^{p^{2mn}}={{(\alpha_{1}}+\theta)}^{p^{mn}}={\alpha_{1}}^{p^{mn}}+\theta^{p^{mn}}=
\alpha_{1}+\theta+\theta=\alpha_{1}+2\theta.$$
 It is easy to show that
$${\alpha_{1}}^{p^{jmn}}={\alpha_{1}}+j\theta, ~~ for ~~ 1\leq j\leq p-1,$$
or
\begin{equation}
{\left(\frac{1}{{\alpha_{1}}}\right)}^{p^{jmn}}=\frac{1}{{\alpha_{1}}+j\theta},
~~ for ~~ 1\leq j\leq p-1.
 \label{eq3344sa}
\end{equation}
From (\ref{eq339sa}) and (\ref{eq3344sa}), we immediately obtain
\begin{equation}
L_{H_{s,i}}(\beta_{1})=\sum_{v=0}^{n-s}{t_{s,i,v}\left({\sum_{j=0}^{p-1}{{\frac{1}{\alpha_{1}+j\theta}}}}\right)^{p^{mv}}}.
\label{eq3345sa}
\end{equation}

Thus, by (\ref{eq3340sa}), (\ref{eq3345sa}) and Lemma \ref{lem21} we
have
\begin{align}
L_{H_{s,i}}(\beta_{1})&=\sum_{v=0}^{n-s}{t_{s,i,v}\left(-\frac{1}{\alpha_{1}^{p}-\alpha_{1}}\right)^{p^{mv}}}\nonumber\\
&=\frac{1}{\delta}\sum_{v=0}^{n-s}{t_{s,i,v}{({1-\alpha})}^{p^{mv}}}\nonumber\\
&=L_{\phi_{s,i}}\left(\frac{1-\alpha}{\delta}\right).
 \label{eq3351sa}
\end{align}
Since $\alpha$ is a zero of $P(x)$, then $\alpha$ will be a
$k$-normal element in $\mathbb{F}_{q^n}$ over $\mathbb{F}_{q}$. Thus
according to Proposition \ref{thm321k}, the element
$\frac{1-\alpha}{\delta}$ will also be a $k$-normal element.

  since
$\frac{1-\alpha}{\delta}$ is a root of  $P(-\delta x+1)$, so by
(\ref{eq3351sa}) and Proposition \ref{pro3122}, there is a $j$,
$1\leq j\leq u_{k}$, such that

$$L_{H_{k,{j}}}(\beta_{1})=0, $$
and also
$$L_{H_{s,{i}}}(\beta_{1})\neq 0,$$
for each $s$, $k<s<n$ and $1\leq i \leq u_{s} $. So, there is a
$j'$, $1\leq j'\leq u'_{k}$, such that,
$L_{H'_{k,{j'}}}(\beta_{1})=L_{H_{k,{j}}}(\beta_{1})=0$. On the
other side, by (\ref{eq31w11ds}) and (\ref{555ss}), for each $s'$,
$k<s'<np$ and $1\leq i' \leq u'_{s'}$, there is $s$, $k<s<n$ and
$1\leq i \leq u_{s}$ such that $H'_{s',{i'}}(x)$ divide
$H_{s,{i}}(x)$. It follows that
$$L_{H'_{s',{i'}}}(\beta_{1})\neq 0,$$
for each $s'$, $k<s'<np$ and $1\leq i' \leq u'_{s'}$. The proof is
completed.
\end{proof}
 In the following theorem, a
computationally simple and explicit recurrent method for
constructing higher degree $N_{k}$-polynomials over $\mathbb{F}_{q}$
starting from an $N_{k}$-polynomial is described.

\begin{theorem}\label{the311ds}
Let $P(x)$ be an $N_k$-polynomial of degree $n$ over
$\mathbb{F}_{q}$, for each $n=rp^e$, where $e\in \mathbb{N}$ and $r$
equals 1 or is a prime different from $p$ and $q$ a primitive
element modulo $r$. Define
$$F_{0}(x)=P^{*}(x)\hspace{4.9cm}$$
\begin{equation}
F_{u}(x)={(x^p-x+\delta)}^{np^{u-1}}F_{u-1}\left(\frac{x^p-x}{x^p-x+\delta}\right),
\label{eq3353}
\end{equation}
where $\delta\in\mathbb{F}^*_{p}$. Then ${(F^{*}_{u}(x))}_{u\geq0}$
is a sequence of $N_k$-polynomials of degree $np^{u}$ over
$\mathbb{F}_{q}$ if $k< p^e$ and
$$ Tr_{q|{p}}\left(\frac{{P^{*}}^{'}(0)}{{P^{*}}(0)}\right)\cdot Tr_{q|{p}}\left(\frac{{P^{*}}^{'}(1)}{P^{*}(1)}\right)\neq0,$$
where ${P^{*}}^{'}(0)$ and ${P^{*}}^{'}(1)$ are the formal
derivative of $P^{*}(x)$ at the points $x=0$ and $x=1$,
respectively.
\end{theorem}

\begin{proof} By Proposition \ref{the4b} and hypotheses of theorem for each
$u\geq1$, $F_{u}(x)$ is an irreducible polynomial over
$\mathbb{F}_{q}$. Consequently, ${(F^{*}_{u}(x))}_{u\geq0}$ is a
sequence of irreducible polynomials over $\mathbb{F}_{q}$. The proof
of $k$-normality of the irreducible polynomials $F^{*}_{u}(x)$, for
each $u\geq1$ is implemented by mathematical induction on $u$. In
the case $u=1$,  by Theorem \ref{thm32uu3} ${F_{1}}^{*}(x)$ is a
$k$-normal polynomial.

For $u=2$ we show that ${F_{2}}^{*}(x)$ is also a $k$-normal
polynomial. To this end we need to show that the hypothesis of
Theorem \ref{thm32uu3} are satisfied. By Theorem \ref{thm32uu3},
${F_{2}}^{*}(x)$ is a $k$-normal polynomial if
$$ Tr_{q|{p}}\left(\frac{F'_{1}(1)}{F_{1}(1)}\right)\neq0,$$

since $\delta\in\mathbb{F}^*_{p}$. We apply (\ref{eq3353}) to
compute
\begin{equation}
F_{u}(0)=F_{u}(1)=\delta^{un}P^{*}(0),~~ u=1,2,\ldots.
 \label{eq3354}
\end{equation}
We calculate the formal derivative of $F'_{1}(x)$ at the points
$x=0$ and $x=1$. According to (\ref{eq3353}) the first derivative of
$F_{1}(x)$ is
\begin{align}
{F_{1}}^{'}(x)&=-n{(x^p-x+\delta)}^{n-1}{F_{0}}^{'}\left(\frac{x^p-x}{x^p-x+\delta}\right)\nonumber\\
\qquad&+{(x^p-x+\delta)}^{n}\cdot\left(\frac{(px^{p-1}-1)(x^p-x+\delta)-(px^{p-1}-1)(x^p-x)}{{(x^p-x+\delta)}^{2}}\right)\nonumber\\
&\cdot{F_{0}}^{'}\left(\frac{x^p-x}{x^p-x+\delta}\right)\nonumber\\
&=-\delta{(x^p-x+\delta)}^{n-2}\cdot{P^{*}}^{'}\left(\frac{x^p-x}{x^p-x+\delta}\right),\nonumber
\label{eq3355}
\end{align}
and at the points $x=0$ and $x=1$
\begin{equation}
{F_{1}}^{'}(0)=-{F_{1}}^{'}(1)=-{\delta}^{n-1}{P^{*}}^{'}(0)
\label{eq3356}
\end{equation}
which is not equal to zero by the condition
$Tr_{q|{p}}\left(\frac{{P^{*}}^{'}(0)}{{P^{*}}(0)}\right)\neq0$ in
the hypothesis of theorem, since $\delta\in\mathbb{F}^*_{p}$. From
(\ref{eq3356}) and (\ref{eq3354})
\begin{align}
Tr_{q|{p}}\left(\frac{F'_{1}(1)}{F_{1}(1)}\right)&=
Tr_{q|{p}}\left(\frac{-{\delta}^{n-1}{P^{*}}^{'}(0)}{\delta^{n}P^{*}(0)}\right)\nonumber\\
&=-\frac{1}{\delta}Tr_{q|{p}}\left(\frac{{P^{*}}^{'}(0)}{P^{*}(0)}\right),
\label{eq3357}
\end{align}
which is not equal to zero by hypothesis of theorem. Hence the
polynomial ${F^{*}_{2}}(x)$ is a $k$-normal polynomial.
 If induction
holds for $u-1$, then it must hold also for $u$, that is by assuming
that $F^{*}_{u-1}(x)$ is a $k$-normal polynomial, we
show that $F^{*}_{u}(x)$ is also a $k$-normal polynomial.\\
Let $u\geq3$. By Theorem \ref{thm32uu3}, $F^{*}_{u}(x)$ is a
$k$-normal polynomial if
$$ Tr_{q|{p}}\left(\frac{F'_{u-1}(1)}{F_{u-1}(1)}\right)\neq0,$$

 since
$\delta\in\mathbb{F}^*_{p}$. We calculate the formal derivative of
$F'_{u-1}(x)$ at the points $1$ and $0$. By (\ref{eq3353}) the first
derivative of $F_{u-1}(x)$ is
\begin{align}
F'_{u-1}(x)&={(x^p-x+\delta)}^{np^{u-2}}\left(\frac{(px^{p-1}-1)(x^p-x+\delta)-(px^{p-1}-1)(x^p-x)}{{(x^p-x+\delta)}^{2}}\right)\nonumber\\
&\cdot F'_{u-2}\left(\frac{x^p-x}{x^p-x+\delta}\right)\nonumber\\
&=-\delta{(x^p-x+\delta)}^{np^{u-2}-2}F'_{u-2}\left(\frac{x^p-x}{x^p-x+\delta}\right),\nonumber
\label{eq3358b}
\end{align}

and at the point $x=0$ and $x=1$
\begin{equation}
F'_{u-1}(0)=F'_{u-1}(1)=-{\delta}^{np^{u-2}-1}F'_{u-2}(0)=-{\delta}^{n-1}F'_{u-2}(0).
\label{eq3358}
\end{equation}
So we have
\begin{equation}
F'_{u-1}(0)=F'_{u-1}(1)={(-1)}^{u-2}{\delta}^{(n-1)(u-2)}F'_{1}(0),
\label{eq3359}
\end{equation}
which is not equal to zero by (\ref{eq3356}) and the condition
$Tr_{q|{p}}\left(\frac{{P^{*}}^{'}(0)}{{P^{*}}(0)}\right)\neq0$ in
the hypothesis of theorem, since $\delta\in\mathbb{F}^*_{p}$. Also
\begin{equation}
Tr_{q|{p}}\left(\frac{F'_{u-1}(1)}{F_{u-1}(1)}\right)={(-1)}^{u-2}\frac{1}{{\delta}^{(u-2)}}Tr_{q|{p}}\left(\frac{F'_{1}(1)}{F_{1}(1)}\right),\nonumber
\end{equation}
which is not equal to zero by (\ref{eq3357}) and hypothesis of
theorem. The theorem is proved.
\end{proof}

\end{document}